\documentclass[reqno]{amsart}
\usepackage{amsaddr}
\usepackage{color}

\usepackage{overpic}
\usepackage{graphics} 
\usepackage{epsfig} 
\usepackage{graphicx}  \usepackage{epstopdf}
\usepackage{subcaption}

\graphicspath{{figures/}}

\newtheorem{theorem}{Theorem}[section]
\newtheorem{lemma}[theorem]{Lemma}
\newtheorem{corollary}[theorem]{Corollary}

\theoremstyle{definition}
\newtheorem{definition}[theorem]{Definition}

\theoremstyle{remark}

\numberwithin{equation}{section}

\newcommand{\authorfootnotes}{\renewcommand\thefootnote{\@fnsymbol\c@footnote}}%

\begin{document}

\title[Nonlocal-to-local convergence of the $p$-Biharmonic equation]
{Nonlocal-to-local convergence of the $p$-Biharmonic evolution equation with the Dirichlet boundary Condition}





\author{Kehan Shi $^{\ast}$, Yi Ran $^{\dagger}$}
\address{$^{\ast}$Department of Mathematics, China Jiliang University, Hangzhou 310018, China}
\address{$^{\dagger}$School of Mathematics, Harbin Institute of Technology, Harbin 150001, China}

\thanks{Yi Ran is the corresponding author: yi.ran@hit.edu.cn}



\subjclass[2010]{45G10,  35G20}

\keywords{$p$-biharmonic equation, Nonlocal equation, Dirichlet boundary condition, Nonlocal-to-local convergence}

\date{}

\dedicatory{}

\begin{abstract}
This paper studies the nonlocal $p$-biharmonic evolution equation with the Dirichlet boundary condition that arises in image processing and data analysis. We prove the existence and uniqueness of solutions to the nonlocal equation and discuss the large time behavior of the solution.
By appropriately rescaling the nonlocal kernel, we further show that the solution converges to the solution of the classical $p$-biharmonic equation with the Dirichlet boundary condition.
Numerical experiments are presented to demonstrate the effectiveness of the nonlocal $p$-biharmonic equation for image inpainting.
\end{abstract}

\maketitle

\section{Introduction}
In this paper, we study the nonlocal $p$-biharmonic evolution equation with the homogeneous Dirichlet boundary condition
\begin{align}\label{eq:1.1}
\left\{
  \begin{array}{ll}
  \frac{\partial u}{\partial t}=-\Delta_{NL}(|\Delta_{NL}u|^{p-2}\Delta_{NL}u), \quad (x,t)\in \Omega\times (0,T), \\
  u(x,t)=0,  \quad (x,t)\in \mathbb{R}^N\setminus \overline{\Omega}\times (0,T), \\
  u(x,0)=u_0(x),\quad x\in\Omega,
  \end{array}
\right.
\end{align}
where $\Omega\subset\mathbb{R}^N (N\geq 2)$ is a bounded domain, $T>0$, and $1<p<+\infty$.
The nonlocal Laplace operator is defined as
\begin{equation*}
  \Delta_{NL}u(x)=\int_{\mathbb{R}^N}J(x-y)\left(u(y)-u(x)\right)dy
  =\int_{\Omega_E}J(x-y)\left(u(y)-u(x)\right)dy,
\end{equation*}
and the nonlocal kernel $J$ is a nonnegative continuous radial function with compact support $\mbox{supp}(J)$. Here $\Omega_E\subset \mathbb{R}^N$ is a sufficiently large domain such that
$\{\Omega+2\mbox{supp}(J)\}\subset \Omega_E$.
Equation \eqref{eq:1.1} is the nonlocal analog of the classical $p$-biharmonic equation
\begin{align}\label{eq:1.2}
\left\{
  \begin{array}{ll}
  \frac{\partial u}{\partial t}=-\Delta(|\Delta u|^{p-2}\Delta u), \quad (x,t)\in \Omega\times (0,T), \\
  u(x,t)=0,\quad \frac{\partial u}{\partial\vec{n}} (x,t)=0,  \quad (x,t)\in \partial\Omega\times (0,T), \\
  u(x,0)=u_0(x),\quad x\in\Omega.
  \end{array}
\right.
\end{align}
Since the operator in equation \eqref{eq:1.1} is nonlocal, it requires to prescribe the value of $u$ outside the domain $\Omega$ for the Dirichlet boundary condition.
Moreover, unlike the classical fourth-order equation \eqref{eq:1.2} which require two boundary conditions, the nonlocal fourth-order equation \eqref{eq:1.1} involves only one boundary condition, as the other boundary condition (analogous to the Neumann boundary condition $\frac{\partial u}{\partial\vec{n}} =0$ on $\partial\Omega$) is implicitly contained.
In fact, it is possible to deduce the homogeneous nonlocal Neumann boundary condition from the boundary condition in \eqref{eq:1.1} in the framework of the nonlocal calculus \cite{du2013nonlocal,radu2017nonlocal}.

Recently, higher-order nonlocal equations of the form of \eqref{eq:1.1} have emerged as an valuable tool for image processing and data analysis.
In \cite{wen2023nonlocal}, the authors propose the nonlocal biharmonic equation
\begin{align}\label{eq:1.3}
\left\{
  \begin{array}{ll}
  \frac{\partial u}{\partial t}=-\Delta_{NL}(\Delta_{NL}u), \quad (x,t)\in \Omega_E\times (0,T), \\
u(x,0)=u_0(x),\quad x\in\Omega_E,
  \end{array}
\right.
\end{align}
 for image denoising, where $u_0$ is the given noisy image and the solution $u(x,t)$ corresponds to the restored image with the scale parameter $t$.
It is a higher order generalization of the nonlocal heat equation \cite{cortazar2008approximate}
\begin{align*}
\left\{
  \begin{array}{ll}
  \frac{\partial u}{\partial t}=\Delta_{NL}u, \quad (x,t)\in \Omega_E\times (0,T), \\
u(x,0)=u_0(x),\quad x\in\Omega_E,
  \end{array}
\right.
\end{align*}
inspired by higher order partial differential equations (PDEs) in image processing \cite{lysaker2003noise,you2000fourth}.
Nonlocal equations share many properties with classical PDEs, but they have no regularizing
effect \cite{chasseigne2006asymptotic}. This is particularly desirable in applications such as image processing, where discontinuous solutions are often preferred \cite{shi2021image,shi2021coupling}.
Equation \eqref{eq:1.3} inherits the advantages of both nonlocal and higher-order methods, preserving texture information in images while reducing contrast loss.
The property of equation \eqref{eq:1.3} has been well studied.
In \cite{wen2023nonlocal}, the authors prove the existence and uniqueness of solutions for equation \eqref{eq:1.3} and show that the solution preserves the mean value for image denoising.
For nonlocal problems, an important property is the connection with the classical analogue.
In \cite{shi2023nonlocal}, the authors discuss the nonlocal-to-local convergence of the problem with the Dirichlet boundary condition and the Navier boundary condition.
It reveals that the nonlocal equation \eqref{eq:1.1} (after appropriately rescale the nonlocal kernel $J$) is a non-smooth approximation to the classical equation \eqref{eq:1.2} when $p=2$.

The linear case $p=2$ is easy to solve numerically. Nonetheless, nonlinear equations, i.e., $p\neq 2$, are often essential in many applications.
The nonlocal $p$-biharmonic operator is considered for image restoration under Cauchy noise \cite{bendaida2025nonlocal}.
We mention that the graph $p$-biharmonic operator has been used in data analysis \cite{dong2020cure,el2020discrete}. It can be considered as the discretization of the nonlocal $p$-biharmonic operator.

The purpose of this paper is to extend some of the results for equation \eqref{eq:1.2} to the nonlocal setting, including the existence and uniqueness of solutions, as well as the long-time behavior as $t\rightarrow\infty$. In addition, we establish the connection between the two equations by proving that the solution of equation \eqref{eq:1.1} converges to the solution of equation \eqref{eq:1.2}. This convergence requires an appropriate rescaling of the nonlocal kernel $J$.
This paper extends the results of \cite{shi2023nonlocal} from the case $p=2$ to the more general case $1<p<\infty$.

For notational simplicity, we restrict our attention to homogeneous boundary conditions for equation \eqref{eq:1.1}.
In practical applications in image processing and data analysis, it is common to consider the nonhomogeneous equation
\begin{align}\label{eq:nonhomogeneous}
\left\{
  \begin{array}{ll}
  \frac{\partial u}{\partial t}=-\Delta_{NL}(|\Delta_{NL}u|^{p-2}\Delta_{NL}u), \quad (x,t)\in \Omega\times (0,T), \\
  u(x,t)=g(x),  \quad (x,t)\in \mathbb{R}^N\setminus \overline{\Omega}\times (0,T), \\
  u(x,0)=u_0(x),\quad x\in\Omega.
  \end{array}
\right.
\end{align}
The results established for equation \eqref{eq:1.1} can be straightforwardly generalized to equation \eqref{eq:nonhomogeneous}.

The rest of this paper is organized as follows. In section 2, we study the existence and uniqueness of solutions for the nonlocal $p$-biharmonic equation and discuss the large time behavior of the solution when $t\rightarrow\infty$.
In section 3, we rescale the nonlocal kernel and prove the convergence of the solution for the rescaled nonlocal equation to the solution for the classical $p$-biharmonic equation with the Dirichlet boundary condition.
Section 4 is devoted to numerical experiments,
where we discuss the application of the nonlocal equation for image inpainting.
We conclude the paper in section 5.

\section{Properties of the nonlocal equation}
In this section, we study properties of solutions for the nonlocal $p$-biharmonic equation \eqref{eq:1.1}, including the well-posedness and the large time behavior.
They are generalizations of the case of the local $p$-biharmonic equation \cite{liu2006weak}.

Throughout this paper, we assume that $\Omega\subset\mathbb{R}^N$ is a bounded domain,  $Q_T=\Omega\times(0,T)$, $T>0$, and $1<p<+\infty$.
The nonlocal kernel $J: \mathbb{R}^n\rightarrow\mathbb{R}$ is a nonnegative continuous radial function with compact support $\textrm{supp}(J)$.
The Dirichlet boundary condition in equation \eqref{eq:1.1} can be replaced by
\begin{equation*}
  u(x,t)=0,  \quad (x,t)\in \Omega_E\setminus \overline{\Omega}\times (0,T),
\end{equation*}
for a bounded domain $\Omega_E\supset\{(\Omega+2\textrm{supp}(J))\}$.

To handle the Dirichlet boundary, we use the notation
\begin{align*}
  \widetilde{u}(x)=\left\{
  \begin{array}{ll}
  u(x), \quad x\in\Omega, \\
  0,\quad x\in\Omega_E\setminus \overline{\Omega},
  \end{array}
  \right.
\end{align*}
for a function $u:\Omega\rightarrow\mathbb{R}$.
Then equation \eqref{eq:1.1} is equivalent to
\begin{align}\label{eq:2.1}
\left\{
  \begin{array}{ll}
  \frac{\partial u}{\partial t}=-\Delta_{NL}(|\Delta_{NL}\widetilde{u}|^{p-2}\Delta_{NL}\widetilde{u}), \quad (x,t)\in Q_T, \\
  u(x,0)=u_0(x),\quad x\in\Omega.
  \end{array}
\right.
\end{align}
Namely, the Dirichlet boundary condition is implicitly contained in the nonlocal $p$-biharmonic operator.

\begin{definition}
  A solution of equation \eqref{eq:1.1} is a function $u\in C([0,T]; L^2(\Omega))$ with $\frac{\partial u}{\partial t}\in L^2(Q_T)$ that satisfies
  \begin{equation}\label{eq:2.2}
    \frac{\partial u}{\partial t}=-\Delta_{NL}(|\Delta_{NL}\widetilde{u}|^{p-2}\Delta_{NL}\widetilde{u}),
    \quad \mbox{a.e. in~} Q_T,
  \end{equation}
  and $u(x,0)=u_0(x)$ for a.e. $x\in\Omega$.
\end{definition}

Let $1\leq q<\infty$.
it follows from the boundedness of $J$ that
\begin{equation}\label{eq:2.2a}
  \int_{\Omega_E}|\Delta_{NL}\widetilde{u}|^qdx=\int_{\Omega_E}\left|\int_{\Omega_E}J(x-y)({\widetilde{u}}(y)
  -{\widetilde{u}}(x))dy\right|^qdx\leq C\int_{\Omega}|{u}|^qdx.
\end{equation}
Conversely,
by the inequality $2^{1-q}|b|^q\leq |b-a|^q+|a|^q$,
\begin{align}\label{eq:2.2b}
\begin{split}
  \int_{\Omega}&\left|u\right|^qdx\leq C\int_{\Omega_E}\left|\int_{\Omega_E}J(x-y)\widetilde{u}(x)dy\right|^qdx\leq C\int_{\Omega_E}\left|\Delta_{NL}\widetilde{u}\right|^qdx \\
  &+C\int_{\Omega_E}\left|\int_{\Omega_E}J(x-y)\widetilde{u}(y)dy\right|^qdx
  \leq C\int_{\Omega_E}\left|\Delta_{NL}\widetilde{u}\right|^qdx+ C\left|\int_{\Omega}udx\right|^q.
\end{split}
\end{align}
The above two inequalities will be frequently used.

The well-posedness of equation \eqref{eq:1.1} is stated as follows.
\begin{theorem}\label{th:existence}
  Let $u_0\in L^2(\Omega)\cap L^p(\Omega)$. Equation \eqref{eq:1.1} admits a unique solution.
\end{theorem}

\begin{proof}
1.
We prove the existence by the Rothe method.
Let $m$ be a positive integer and $h=\frac{T}{m}$.
For a given $u^{m,j-1}\in L^2(\Omega)$,
consider the semi-discrete equation
  \begin{align}\label{eq:2.3}
  \begin{split}
    \frac{u-u^{m,j-1}}{h}=-\Delta_{NL}(|\Delta_{NL}\widetilde{u}|^{p-2}\Delta_{NL}\widetilde{u}),\quad \mbox{in } \Omega.
  \end{split}
  \end{align}
  By \eqref{eq:2.2a} and Young's inequality, the associated energy functional
  \begin{align*}
    E(w)=\frac{1}{2h}\int_{\Omega}|w|^2dx-\frac{1}{h}\int_{\Omega}u^{m,j-1}wdx
  +\frac{1}{p}\int_{\Omega_E}|\Delta_{NL}\widetilde{w}|^pdx,
  \end{align*}
 is well-defined for any $w\in L^2(\Omega)\cap L^p(\Omega)$ and has a lower bound.
This implies the existence of a minimizing sequence $\{u_{k}\}\subset L^2(\Omega)\cap L^p(\Omega)$, such that
\begin{equation*}
  \lim_{k\rightarrow\infty}E(u_{k})=\inf_{w\in L^2(\Omega)\cap L^p(\Omega)}E(w),
\end{equation*}
and
\begin{equation*}
  \int_{\Omega}|u_{k}|^2dx
  +\int_{\Omega_E}\left|\Delta_{NL}\widetilde{u}_{k}\right|^pdx\leq C\left(E(u_{k})+\int_{\Omega}|u^{m,j-1}|^2dx\right)\leq C.
\end{equation*}
Besides, by \eqref{eq:2.2b},
\begin{equation*}
  \int_{\Omega}\left|{u_{k}}\right|^pdx \leq C.
\end{equation*}
Then there exist a subsequence of $\{u_{k}\}$ (still denoted by itself) and measurable functions $u^{m,j}\in L^2(\Omega)\cap L^p(\Omega)$, $\theta\in L^p(\Omega_E)$, such that
\begin{gather*}
  u_{k} \rightharpoonup u^{m,j},\quad \textrm{in~} L^2(\Omega)\cap L^p(\Omega),\\
  \Delta_{NL}\widetilde{u}_{k} \rightharpoonup \theta, \quad \textrm{in~} L^p(\Omega_E).
\end{gather*}
For any $\varphi\in C_0^\infty(\Omega)$,
by the nonlocal integration by parts,
\begin{align*}
  \int_{\Omega_E}\Delta_{NL}\widetilde{u}_{k}\cdot\widetilde{\varphi} dx=\int_{\Omega_E}\widetilde{u}_{k}\Delta_{NL}\widetilde{\varphi} dx\rightarrow
  \int_{\Omega_E}\widetilde{u}^{m,j}\Delta_{NL}\widetilde{\varphi} dx
  =\int_{\Omega_E}\Delta_{NL}\widetilde{u}^{m,j}\cdot\widetilde{\varphi} dx,
\end{align*}
as $k\rightarrow+\infty$.
It means that $\theta=\Delta_{NL}\widetilde{u}^{m,j}$.
Consequently, by the weakly lower semi-continuity of $E(w)$ on $L^2(\Omega)\cap L^p(\Omega)$, $u^{m,j}$ is a minimizer of $E(w)$ and it satisfies the Euler--Lagrange equation
\begin{equation}\label{eq:2.4}
  \int_{\Omega}\frac{u^{m,j}-u^{m,j-1}}{h}\varphi dx+\int_{\Omega_E}|\Delta_{NL}\widetilde{u}^{m,j}|^{p-2}\Delta_{NL}\widetilde{u}^{m,j}\Delta_{NL}\widetilde{\varphi} dx=0,
\end{equation}
for any $\varphi\in L^2(\Omega)\cap L^p(\Omega)$.

Let $u^{m,0}=u_0$.
By repeating the above procedure for $j=1,2,\cdots,m$, we obtain a sequence of functions $\{u^{m,j}\}_{j=1}^m\subset L^2(\Omega)\cap L^p(\Omega)$ that satisfy identity \eqref{eq:2.4}.

2.
Define approximate solutions
\begin{align*}
  u^m(x,t)&=\sum_{j=1}^{m}\chi^{m,j}(t)u^{m,j-1}(x)
  + \sum_{j=1}^{m}\chi^{m,j}(t)\lambda^{m,j}(t) \left(u^{m,j}(x)-u^{m,j-1}(x)\right),\\
  w^m(x,t)&=\sum_{j=1}^{m}\chi^{m,j}(t)u^{m,j}(x),
\end{align*}
where $\chi^{m,j}$ is the characteristic function on the interval $[(j-1)h, jh)$ and
\begin{equation*}
  \lambda^{m,j}(t)=
  \begin{cases}
    \frac{t}{h}-(j-1), & \mbox{if } t\in[(j-1)h, jh), \\
    0, & \mbox{otherwise}.
  \end{cases}
\end{equation*}
It follows from \eqref{eq:2.4} that
\begin{equation}\label{eq:2.5}
  \int_{\Omega}\frac{\partial u^m}{\partial t}\varphi dx+\int_{\Omega_E}|\Delta_{NL}\widetilde{w}^m|^{p-2}\Delta_{NL}\widetilde{w}^m\Delta_{NL}\widetilde{\varphi} dx=0,
\end{equation}
for any $t\in (0,T)$.

Let $\varphi=u^{m,j}-u^{m,j-1}$ in \eqref{eq:2.4}. It yields that
\begin{equation*}
  \frac{1}{h}\int_{\Omega}|u^{m,j}-u^{m,j-1}|^2dx+
  \frac{1}{p}\int_{\Omega_E}|\Delta_{NL}\widetilde{u}^{m,j}|^pdx
  \leq  \frac{1}{p}\int_{\Omega_E}|\Delta_{NL}\widetilde{u}^{m,j-1}|^pdx.
\end{equation*}
Summing over $j=1,\cdots m$, we obtain
\begin{equation}\label{eq:2.6}
  \frac{1}{h}\sum_{j=1}^{m}\int_{\Omega}|u^{m,j}-u^{m,j-1}|^2dx+
  \frac{1}{p}\int_{\Omega_E}|\Delta_{NL}\widetilde{u}^{m,j}|^pdx
  \leq  \frac{1}{p}\int_{\Omega_E}|\Delta_{NL}\widetilde{u}_0|^pdx.
\end{equation}
Consequently,
\begin{align*}
  \iint_{Q_T}\left|\frac{\partial u^m}{\partial t}\right|^2dxdt
  &=\iint_{Q_T}\left|\frac{1}{h}\chi^{m,j}\left(u^{m,j}-u^{m,j-1}\right)\right|^2dxdt\\
  &=\frac{1}{h^2}\sum_{j=1}^{m}h\int_\Omega\left|u^{m,j}-u^{m,j-1}\right|^2dx\leq C,
\end{align*}
and
\begin{equation*}
  \sup_{t\in[0,T]}\int_{\Omega_E}
  |\Delta_{NL}\widetilde{w}^m|^pdx
  \leq C.
\end{equation*}
According to the definition of $u^m$ and $w^m$,
\begin{align*}
  \iint_{Q_T}\left|w^m-u^m\right|^2dxdt
  &=\iint_{Q_T}\left|\sum_{j=1}^{m}\chi^{m,j}(1-\lambda^{m,j})(u^{m,j}-u^{m,j-1})\right|^2dxdt\\
  &\leq h\sum_{j=1}^{m}\int_{\Omega}\left|u^{m,j}-u^{m,j-1}\right|^2dxdt\rightarrow 0,\quad \mbox{as } m\rightarrow \infty.
\end{align*}
Let $\varphi=w^{m}$ in \eqref{eq:2.5}. We further have
\begin{equation*}
  \sup_{t\in[0,T]}\int_{\Omega}|u^m|^2dx\leq C.
\end{equation*}
All the above estimates allow us to extract subsequences of $\{u^m\}$ and $\{w^m\}$ (still denoted by themselves), such that
\begin{align*}
  u^m\stackrel{*}{\rightharpoonup}& u, \quad \textrm{in}~ L^\infty((0,T);L^2(\Omega)), \\
  \frac{\partial u^m}{\partial t}\rightharpoonup& \frac{\partial u}{\partial t}, \quad \textrm{in}~ L^2(Q_T),\\
  w^m\rightharpoonup& u, \quad \textrm{in}~ L^2(Q_T),\\
  |\Delta_{NL}\widetilde{w}^m|^{p-2}\Delta_{NL}\widetilde{w}^m\rightharpoonup & \vartheta, \quad \textrm{in}~ L^{p/(p-1)}(\Omega_E\times (0,T)),
\end{align*}
for measurable functions $u\in L^\infty((0,T);L^2(\Omega))$ and $\vartheta\in L^{p/(p-1)}(\Omega_E\times (0,T))$.
Furthermore, $u\in C([0,T];L^2(\Omega))$ and $u(x,0)=u_0(x)$ for a.e. $x\in\Omega$.

3. To show that $u$ is a solution of equation \eqref{eq:2.1}, we are left to verify identity \eqref{eq:2.2}.
Passing to the limit $m\rightarrow \infty$ in \eqref{eq:2.5} and integrating over $(0,T)$ yield
\begin{align}\label{eq:2.7}
  \begin{split}
     \iint_{Q_T}\frac{\partial u}{\partial t}\varphi dxdt+\int_0^{T}\int_{\Omega_E}\vartheta \cdot \Delta_{NL}\widetilde{\varphi} dxdt=0,
  \end{split}
\end{align}
for any $\varphi\in L^2(\Omega)\cap L^p(\Omega)$.
We show $\vartheta=|\Delta_{NL}\widetilde{u}|^{p-2}\Delta_{NL}\widetilde{u}$ in the following.

It follows from the monotonicity of $|s|^{p-2}s$ that
\begin{align*}
\int_{\Omega_E}\left(|\Delta_{NL}\widetilde{u}^{m,j}|^{p-2}\Delta_{NL}\widetilde{u}^{m,j}
  -|\Delta_{NL}\widetilde{v}|^{p-2}\Delta_{NL}\widetilde{v}\right)\cdot\Delta_{NL}(\widetilde{u}^{m,j}-\widetilde{v}) dx
  \geq 0,
\end{align*}
for any $v\in L^p(\Omega_E)$.
Taking $\varphi=u^{m,j}$ in \eqref{eq:2.4}, we arrive at
\begin{align*}
     \int_{\Omega}\frac{{u}^{m,j}-u^{m,j-1}}{h}{u}^{m,j} dx
     +\int_{\Omega_E}|\Delta_{NL}\widetilde{u}^{m,j}|^{p} dx=0.
\end{align*}
Combining it with the above monotonicity and integrating over $[(j-1)h, jh]$ lead to
\begin{align*}
     \int_{\Omega}\frac{|u^{m,j}|^2-|u^{m,j-1}|^2}{2} dx
     &+\int_{(j-1)h}^{jh}\int_{\Omega_E}|\Delta_{NL}\widetilde{u}^{m,j}|^{p-2}\Delta_{NL}\widetilde{u}^{m,j}\cdot\Delta_{NL}\widetilde{v} dx\\
     &+\int_{(j-1)h}^{jh}\int_{\Omega_E}|\Delta_{NL}\widetilde{v}|^{p-2}\Delta_{NL}\widetilde{v}\cdot\Delta_{NL}(\widetilde{u}^{m,j}
     -\widetilde{v}) dx\leq 0.
\end{align*}
Summing up the above inequalities for $j=1,2,\cdots,m$ and
passing to the limit $m\rightarrow \infty$, we have
\begin{align*}
     \int_{\Omega}\frac{|u(x,T)|^2-|u_0|^2}{2} dx&+\int_0^{T}\int_{\Omega_E}\vartheta\cdot\Delta_{NL}\widetilde{v} dxdt\\
     &+\int_0^{T}\int_{\Omega_E}|\Delta_{NL}\widetilde{v}|^{p-2}\Delta_{NL}\widetilde{v}\cdot\Delta_{NL}(\widetilde{u}
     -\widetilde{v}) dxdt \leq 0.
\end{align*}
Substituting the above result into \eqref{eq:2.7} with $\varphi=u$, we obtain
\begin{align*}
  \int_0^{T}\int_{\Omega_E}\left(|\Delta_{NL}\widetilde{v}|^{p-2}\Delta_{NL}\widetilde{v}-\vartheta\right)
  \Delta_{NL}(\widetilde{u}-\widetilde{v}) dxdt\leq 0.
\end{align*}
Taking $v=u-\alpha\varphi$ with $\alpha>0$, $\varphi\in C_0^\infty(Q_T)$ and let $\alpha\rightarrow0^+$, we obtain
\begin{align*}
  \int_0^{T}\int_{\Omega_E}\left(|\Delta_{NL}\widetilde{u}|^{p-2}\Delta_{NL}\widetilde{u}-\vartheta\right)
  \Delta_{NL}\widetilde{\varphi} dxdt\leq 0.
\end{align*}
Similarly, we get an opposite inequality by taking $\alpha< 0$.
Consequently,
\begin{align*}
  \int_0^{T}\int_{\Omega_E}\left(|\Delta_{NL}\widetilde{u}|^{p-2}\Delta_{NL}\widetilde{u}-\vartheta\right)
  \Delta_{NL}\widetilde{\varphi} dxdt= 0.
\end{align*}
Now we return to \eqref{eq:2.7} to find
\begin{align*}
     \iint_{Q_T}\frac{\partial u}{\partial t}\varphi dxdt&=-\int_0^{T}\int_{\Omega_E}|\Delta_{NL}\widetilde{u}|^{p-2}\Delta_{NL}\widetilde{u} \cdot\Delta_{NL}\widetilde{\varphi} dxdt\\
     &=-\int_0^{T}\int_{\Omega}\Delta_{NL}\left(|\Delta_{NL}\widetilde{u}|^{p-2}\Delta_{NL}\widetilde{u}\right) \varphi dxdt.
\end{align*}
for any $\varphi\in C_0^\infty(Q_T)$.
 This completes the proof of \eqref{eq:2.2}.

4.
To prove the uniqueness of the solution, we let $u_{1}$ and $u_{2}$ be two solutions. Then it follows from \eqref{eq:2.2} that
\begin{align*}
  &\int_0^{t}\int_{\Omega}\frac{\partial (u_{1}-u_{2})}{\partial t}\varphi dxd\tau \\ &+\int_0^{t}\int_{\Omega_E}\left(|\Delta_{NL}\widetilde{u}_1|^{p-2}\Delta_{NL}\widetilde{u}_1-
  |\Delta_{NL}\widetilde{u}_2|^{p-2}\Delta_{NL}\widetilde{u}_2\right)\Delta_{NL}\widetilde{\varphi} dxd\tau =0.
\end{align*}
for any $t\in [0,T]$.
Choosing $u_{1}-u_{2}$ as the test function in the above and noticing the monotonicity of $|s|^{p-2}s$,
we have
\begin{equation*}
  \int_{\Omega}|u_{1}(x,t)-u_{2}(x,t)|^2dx =2\int_0^{t}\int_{\Omega}\frac{\partial (u_{1}-u_{2})}{\partial t}(u_{1}-u_{2}) dxd\tau\leq 0.
\end{equation*}
The uniqueness is proven.
\end{proof}

Theorem \ref{th:existence} can be generalized to the non-homogeneous equation \eqref{eq:nonhomogeneous}.
The proof is identical, except that the zero extension $\widetilde{u}$ is replaced by the extension associated with the boundary value $g$, i.e.,
  \begin{align*}
  \widetilde{u}_g(x)=\left\{
  \begin{array}{ll}
  u(x), \quad x\in\Omega, \\
  g(x),\quad x\in\Omega_E\setminus \overline{\Omega}.
  \end{array}
  \right.
\end{align*}
\begin{corollary}\label{cor:1}
  Let $u_0\in L^2(\Omega)\cap L^p(\Omega)$ and $g\in L^p((\Omega_E\setminus \overline{\Omega}))$.
  Equation \eqref{eq:nonhomogeneous} admits a unique solution. More precisely,
   there exists a unique function $u\in C([0,T]; L^2(\Omega))$ with $\frac{\partial u}{\partial t}\in L^2(Q_T)$, such that
  \begin{equation}
    \frac{\partial u}{\partial t}=-\Delta_{NL}(|\Delta_{NL}\widetilde{u}_g|^{p-2}\Delta_{NL}\widetilde{u}_g),
    \quad \mbox{a.e. in~} Q_T,
  \end{equation}
  and $u(x,0)=u_0(x)$ for a.e. $x\in\Omega$.
\end{corollary}

The nonlocal $p$-biharmonic equation \eqref{eq:1.1} has the same asymptotic behavior as the local $p$-biharmonic equation when $p\geq 2$.
It is a corollary of the nonlocal Poincar\'{e} inequality \cite[Proposition 6.25]{andreu2010nonlocal}.
\begin{lemma}
  Let $u\in L^q(\Omega)$ and $q\geq 1$. There exists a constant $C=C(J,\Omega,p)>0$, such that
  \begin{equation}\label{eq:2.8}
    \int_{\Omega}\left|u\right|^qdx\leq C\int_\Omega\int_{\Omega_E} J(x-y)|\tilde{u}(y)-u(x)|^qdydx.
  \end{equation}
\end{lemma}

\begin{theorem}
  Let $p\geq 2$ and $u$ be the solution of equation \eqref{eq:1.1}.
  Then for any $t>0$, we have
  \begin{equation*}
    \int_{\Omega}|u(x,t)|^2dx\leq
    \begin{cases}
      e^{-C_1t}\int_{\Omega}|u_0(x)|^2dx, & \mbox{if } p=2, \\
      \frac{1}{\left(C_2t+C_3\right)^{2/(p-2)}}, & \mbox{if } p>2.
    \end{cases}
  \end{equation*}
  Here $C_1, C_2, C_3$ are positive constants.
\end{theorem}

\begin{proof}
Let $u$ be the test function for equation \eqref{eq:2.2}. We have
\begin{equation*}
  \frac{1}{2}\int_{\Omega}|u(x,t)|^2dx-\frac{1}{2}\int_{\Omega}|u_0(x)|^2dx=-\int_0^t\int_{\Omega_E}
  |\Delta_{NL}\widetilde{u}|^pdxdt.
\end{equation*}
Define $f(t)=\int_{\Omega}|u(x,t)|^2dx$. It says that
\begin{equation*}
  f'(t)=-\int_{\Omega_E}|\Delta_{NL}\widetilde{u}|^pdx\leq 0.
\end{equation*}
Consequently, by \eqref{eq:2.8} and the nonlocal integration by parts,
\begin{align*}
  f(t)=\int_{\Omega}|u(x,t)|^2dx
  &\leq C\int_\Omega\int_{\Omega_E} J(x-y)|\tilde{u}(y)-u(x)|^2dydx\\
  &= C\int_{\Omega_E}\Delta_{NL}\widetilde{u}\cdot\widetilde{u}dx
  \leq C\int_{\Omega_E}|\Delta_{NL}\widetilde{u}|^2dx + \frac{1}{2}f(t),
\end{align*}
which implies that
\begin{align*}
   f(t) \leq C\int_{\Omega_E}|\Delta_{NL}\widetilde{u}|^2dx
  \leq C\left(\int_{\Omega_E}
  |\Delta_{NL}\widetilde{u}|^pdx\right)^{2/p}  = C|f'(t)|^{2/p},
\end{align*}
Since $f'(t)\leq 0$, we have $f'(t)\leq -Cf(t)^{p/2}$.
If $p=2$, the result follows from Gronwall's inequality.
If $p>2$, we let $h(t)=e^{\frac{2}{2-p}f(t)^{(2-p)/2}}$. Then $h(t)$ satisfies $h'(t)\leq -Ch$.
Again, we utilize Gronwall's inequality to obtain the result.
\end{proof}

\section{Nonlocal-to-local convergence}
In this section, we consider the rescaled nonlocal equation
\begin{align}\label{eq:4.1}
\left\{
  \begin{array}{ll}
  \frac{\partial u}{\partial t}=-\Delta_{NL}^{J_\varepsilon}(|\Delta_{NL}^{J_\varepsilon}u|^{p-2}\Delta_{NL}^{J_\varepsilon}u), \quad (x,t)\in Q_T, \\
  u(x,t)=0,  \quad (x,t)\in \Omega_E\setminus \overline{\Omega}\times (0,T), \\
  u(x,0)=u_0(x),\quad x\in\Omega,
  \end{array}
\right.
\end{align}
where
\begin{equation*}
    \Delta_{NL}^{J_\varepsilon}u(x,t)=\int_{\Omega_E}J_\varepsilon(x-y)\left(u(y,t)-u(x,t)\right)dy,
\end{equation*}
and
\begin{equation*}
  J_{\varepsilon}(x)=\frac{C_{J}}{\varepsilon^{N+2}}J\left(\frac{x}{\varepsilon}\right),\quad C^{-1}_{J}=\frac{1}{2}\int_{\mathbb{R}^N}J(z)|z|^2dz.
\end{equation*}
Here and in the following $J$ is assumed to be nonincreasing, i.e., $J(x)\geq J(y)$ for $|x|\leq |y|$.

The goal is to establish the connection between equation \eqref{eq:4.1} and the classical $p$-biharmonic evolution equation with the Dirichlet boundary condition
\begin{align}\label{eq:4.2}
\left\{
  \begin{array}{ll}
  \frac{\partial u}{\partial t}=-\Delta(|\Delta u|^{p-2}\Delta u), \quad (x,t)\in Q_T, \\
  u(x,t)=0,\quad \frac{\partial u}{\partial\vec{n}} (x,t)=0,  \quad (x,t)\in \partial\Omega\times (0,T), \\
  u(x,0)=u_0(x),\quad x\in\Omega.
  \end{array}
\right.
\end{align}
In the following the boundary $\partial \Omega$ is assumed to be $C^2$.
We recall the unique solvability of equation \eqref{eq:4.2}, whose proof is similar to that of nonlocal equation \eqref{eq:1.1}.

\begin{lemma}\label{le:4.1}
  Let $u_0\in W_0^{2,p}(\Omega)\cap L^2(\Omega)$. Then equation \eqref{eq:4.2} admits a unique weak solution
  $u\in C([0,T]; L^2(\Omega))\cap L^\infty((0,T); W_0^{2,p}(\Omega))$, such that
\begin{align}\label{eq:4.3}
\begin{split}
    -\iint_{Q_T}u\frac{\partial \varphi}{\partial t}dxdt +\iint_{Q_T}|\Delta u|^{p-2}\Delta u \Delta\varphi dxdt =0,
  \end{split}
\end{align}
for any $\varphi\in C_0^\infty(Q_T)$ and $u(x,0)=u_0(x)$ a.e. in $\Omega$.
\end{lemma}

Let $\varepsilon>0$ be fixed and $u_\varepsilon$ be the solution of the nonlocal equation \eqref{eq:4.1}. We shall prove the convergence of $u_\varepsilon$ to the solution $u$ of the local equation \eqref{eq:4.2} as $\varepsilon\rightarrow 0$.
Notice that the $p$-biharmonic operator takes the Laplacian as the ingredient. The proof relies on the estimates of the Laplacian operator and the Poisson equation.

\begin{lemma}
If $\varphi\in W^{2,p}(\Omega)$, then
  \begin{equation}\label{eq:4.4}
  \lim_{\varepsilon\rightarrow 0}\left\|\Delta_{NL}^{J_\varepsilon}\hat{\varphi}-\Delta \varphi\right\|_{L^p(\Omega)}= 0,
\end{equation}
where $\hat{\varphi}\in W^{2,p}(\mathbb{R}^N)$ is any extension of $\varphi$.
If $\varphi_\varepsilon, \varphi\in W_0^{1,p}(\Omega)\cap W^{2,p}(\Omega)$
 and $\varphi_\varepsilon \rightharpoonup \varphi$ in $L^p(\Omega)$,  then
  \begin{equation}\label{eq:4.5}
  \Delta_{NL}^{J_\varepsilon}\widetilde{\varphi}_\varepsilon\rightharpoonup\Delta \varphi,
\end{equation}
in ${L^p(\Omega)}$.
\end{lemma}
\begin{proof}
  If $\varphi\in C^{\infty}(\overline{\Omega})$ and $x\in\Omega$,
  by the change of variables $z=(x-y)/\varepsilon$ and the Taylor expansion,
\begin{align*}
  \Delta_{NL}^{J_\varepsilon}\hat{\varphi}(x)
  &=\frac{C_{J}}{\varepsilon^{N+2}}
  \int_{\mathbb{R}^N}J\left(\frac{x-y}{\varepsilon}\right)\left(\hat{\varphi}(y)-\varphi(x)\right)dy\\
  &=\frac{C_{J}}{\varepsilon^{2}}
  \int_{\mathbb{R}^N}J(z)\left(\hat{\varphi}(x-\varepsilon z)-\varphi(x)\right)dz \\
  &=\frac{C_{J}}{\varepsilon^{2}}
  \int_{\mathbb{R}^N}J(z)\left(-\varepsilon z_i\varphi_{x_i}(x)+\frac{\varepsilon^2}{2}z_iz_j\varphi_{x_ix_j}(x)+O(\varepsilon^3)\right)dz\\
  &=\Delta\varphi(x)+O(\varepsilon).
\end{align*}
Let $\varphi\in W^{2,p}(\Omega)$. There exists a sequence of functions $\{\varphi_n\}_{n=1}^\infty\subset C^{\infty}(\overline{\Omega})$, such that
\begin{equation*}
  \|\varphi_n-\varphi\|_{W^{2,p}(\Omega)}\leq \varepsilon,
\end{equation*}
for sufficient large $n$.
Then  by
\begin{align*}
  \|\Delta_{NL}^{J_\varepsilon}\hat{\varphi}-\Delta \varphi\|_{L^p(\Omega)}
  \leq\|\Delta_{NL}^{J_\varepsilon}\hat{\varphi}_n-\Delta \varphi_n\|_{L^p(\Omega)}
  +\|\Delta\varphi_n-\Delta\varphi\|_{L^p(\Omega)},
  \end{align*}
we obtain \eqref{eq:4.4}.

For the proof of \eqref{eq:4.5}, We utilize integration by parts, i.e.,
\begin{align*}
  \int_{\Omega}\Delta_{NL}^{J_\varepsilon}\widetilde{\varphi}_\varepsilon\cdot\psi dx
  =\int_{\Omega_E}\widetilde{\varphi}_\varepsilon\cdot\Delta_{NL}^{J_\varepsilon}\widetilde{\psi} dx
  \rightarrow \int_{\Omega_E}\widetilde{\varphi}\cdot\Delta{\psi} dx
  =\int_{\Omega}\Delta{\varphi}\cdot \psi dx,
\end{align*}
for any $\psi\in C_0^\infty(\Omega)$.
\end{proof}

We recall the \textit{a priori} estimates of the solution and its nonlocal gradient for the nonlocal Poisson solution \cite{shi2025continuum}.

\begin{lemma}\label{le:poisson}
 Let $h\in L^p(\Omega_E)$. If $u\in L^1(\Omega)$ is a solution of the nonlocal Poisson equation
  \begin{equation*}
    -\Delta_{NL}^{J_\varepsilon}\widetilde{u}=h,\quad \mbox{in }\Omega_E,
  \end{equation*}
then $u\in L^p(\Omega)$. Furthermore,
\begin{equation*}
  \int_\Omega |u|^pdx+\int_{\Omega_E}\int_{\Omega_E}J_\varepsilon(x-y)\left|\frac{\widetilde{u}(y)-\widetilde{u}(x)}{\varepsilon}\right|^pdydx
  \leq C\left(\int_\Omega|u|dx\right)^p
  +C\int_{\Omega_E}|h|^pdx+C,
\end{equation*}
for $1<p< 2$ and
\begin{align*}
  \int_\Omega |u|^pdx+\int_{\Omega_E}\int_{\Omega_E} J_{\varepsilon}(x-y)\left|\frac{(|\widetilde{u}|^{(p-2)/2}\widetilde{u})(y)
  -(|\widetilde{u}|^{(p-2)/2}\widetilde{u})(x)}{\varepsilon}\right|^2dydx \qquad\\
  \leq C\left(\int_\Omega|u^{p/2}|dx\right)^2 + C\int_{\Omega_E}|h|^pdx,
\end{align*}
for $p\geq2$.
Here the constant $C$ does not depend on $\varepsilon$.
\end{lemma}

The last tool we need is the following compactness result. It first appears in \cite[Proposition 3.2]{andreu2008nonlocal}. The zero trace is straightforward (see, e.g., \cite[Theorem 5.29]{adams2003sobolev}).

\begin{lemma}\label{compact}
  Let $\{u_\varepsilon\}$ be a sequence of functions in $L^p(\Omega)$ such that
  \begin{align*}
  &\frac{1}{\varepsilon^{p-2}}\int_{\Omega_E}\int_{\Omega_E}J_\varepsilon(x-y)
  |\widetilde{u}_\varepsilon(y)-\widetilde{u}_\varepsilon(x)|^pdydx\leq C.
  \end{align*}
  If $u_\varepsilon\rightharpoonup u$ in $L^p(\Omega)$, then $u\in W_0^{1,p}(\Omega)$ and $\{u_\varepsilon\}$  is relatively compact in $L^p(\Omega)$.
\end{lemma}

The main result is stated as follows.
\begin{theorem}\label{th:4.1}
    Let $1<p<\infty$ and $u_0\in W_0^{2,p}(\Omega)\cap L^2(\Omega)$.
    If $u_\varepsilon$ is the solution of the nonlocal equation \eqref{eq:4.1} and $u$ is the weak solution of the local equation \eqref{eq:4.2}, then
    \begin{equation}\label{eq:4.6}
    \lim_{\varepsilon\rightarrow 0}\sup_{t\in[0,T]}\|u_\varepsilon(x,t)-u(x,t)\|_{L^p(\Omega)}=0.
  \end{equation}
\end{theorem}

\begin{proof}
1.  It follows from Theorem \ref{th:existence} that
\begin{equation}\label{eq:4.7}
  \int_0^{t}\int_{\Omega}\frac{\partial u_\varepsilon}{\partial t}\varphi dxd\tau +\int_0^{t}\int_{\Omega_{E}}|\Delta_{NL}^{J_\varepsilon}\widetilde{u}_\varepsilon|^{p-2}
  \Delta_{NL}^{J_\varepsilon}\widetilde{u}_\varepsilon\cdot
  \Delta_{NL}^{J_\varepsilon}\widetilde{\varphi} dxd\tau =0,
\end{equation}
for any $t\in (0,T)$ and any $\varphi\in L^2(\Omega)\cap L^p(\Omega)$.
Let $u_\varepsilon$ and $\frac{\partial u_\varepsilon}{\partial t}$ be the test function in the above respectively, it yields that
\begin{align}\label{eq:4.8}
\begin{split}
  \frac{1}{2}\sup_{t\in[0,T]}\int_{\Omega}|u_\varepsilon|^2dx
  +\iint_{Q_T}\left|\frac{\partial u_\varepsilon}{\partial t}\right|^2dxdt
  +\frac{1}{p}\sup_{t\in[0,T]}\int_{\Omega_E}\left|\Delta_{NL}^{J_\varepsilon}\widetilde{u}_\varepsilon\right|^pdx\\
  \leq  \frac{1}{2}\int_{\Omega}|u_0|^2dx+
  \frac{1}{p}\int_{\Omega_E}\left|\Delta_{NL}^{J_\varepsilon}\widetilde{u}_0\right|^pdx.
\end{split}
\end{align}
By \eqref{eq:4.4}, the right hand side of \eqref{eq:4.8} is bounded.
Lemma \ref{le:poisson} further implies that
\begin{equation}\label{eq:4.9}
  \sup_{t\in[0,T]}\int_{\Omega}\left|{u_\varepsilon}\right|^pdx\leq C,\quad 1<p<\infty,
\end{equation}
and
\begin{equation}\label{eq:4.10a}
  \sup_{t\in[0,T]}\int_{\Omega_E}\int_{\Omega_E}J_\varepsilon(x-y)
  |\widetilde{u}_\varepsilon(y,t)-\widetilde{u}_\varepsilon(x,t)|^pdydx\leq C,
\end{equation}
if $1<p<2$,
\begin{equation}\label{eq:4.10b}
  \sup_{t\in[0,T]}\int_{\Omega_E}\int_{\Omega_E}J_\varepsilon(x-y)
  ||\widetilde{u}_\varepsilon|^{(p-2)/2}\widetilde{u}_\varepsilon(y,t)
  -|\widetilde{u}_\varepsilon|^{(p-2)/2}\widetilde{u}_\varepsilon(x,t)|^pdydx\leq C,
\end{equation}
if $p\geq 2$.

The estimates \eqref{eq:4.8}--\eqref{eq:4.9} indicate that there exists a subsequence of $\{u_\varepsilon\}$ (still denoted by itself), such that
\begin{gather*}
  \frac{\partial u_\varepsilon}{\partial t}\rightharpoonup\frac{\partial u}{\partial t},\quad \textrm{in}~ L^2(Q_T), \\
  u_\varepsilon\stackrel{*}{\rightharpoonup} u, \quad \textrm{in}~L^\infty((0,T); L^p(\Omega)\cap L^2(\Omega)),\\
  \Delta_{NL}^{J_\varepsilon}\widetilde{u}_\varepsilon \stackrel{*}{\rightharpoonup} \theta,\quad \textrm{in}~ L^\infty((0,T); L^p(\Omega_E)), \\
  |\Delta_{NL}^{J_\varepsilon}\widetilde{u}_\varepsilon |^{p-2}\Delta_{NL}^{J_\varepsilon}\widetilde{u}_\varepsilon \stackrel{*}{\rightharpoonup} \vartheta,\quad \textrm{in}~ L^\infty((0,T); L^{p/(p-1)}(\Omega)),
\end{gather*}
for measurable functions $u\in L^\infty((0,T); L^p(\Omega)\cap L^2(\Omega))$, $\theta\in L^\infty((0,T); L^p(\Omega_E))$, and $\vartheta\in L^\infty((0,T); L^{p/(p-1)}(\Omega))$.

2. We first claim that \eqref{eq:4.6} holds and $u\in L^\infty((0,T); W_0^{1,\min\{2,p\}}(\Omega))$.
In fact, if $1<p<2$, it follows from Lemma \ref{compact} and \eqref{eq:4.10a} directly.
If $p\geq 2$, we first let $p=2$ in \eqref{eq:4.10b} and utilize Lemma \ref{compact} to conclude that
$u\in L^\infty((0,T); W_0^{1,2}(\Omega))$ and $u_\varepsilon(\cdot,t)\rightarrow u(\cdot,t)$ in $L^2(\Omega)$.
It allows us to extract a subsequence (still denoted by itself) such that  $u_\varepsilon(\cdot,t)\rightarrow u(\cdot,t)$ a.e. in $\Omega$.
Notice that
\begin{equation*}
  \sup_{t\in[0,T]}\int_\Omega\left| |u_\varepsilon|^{(p-2)/2}u_\varepsilon \right|^2dx
  =\sup_{t\in[0,T]}\int_\Omega|u_\varepsilon|^pdx\leq C.
\end{equation*}
We have
\begin{equation*}
  |u_\varepsilon|^{(p-2)/2}u_\varepsilon
  \stackrel{*}{\rightharpoonup} |u|^{(p-2)/2}u,\quad \textrm{in}~ L^\infty((0,T); L^2(\Omega)).
\end{equation*}
Now applying Lemma \ref{compact} for $|u_\varepsilon|^{(p-2)/2}u_\varepsilon$ and utilizing \eqref{eq:4.10b} once again, we finally arrive at
\begin{equation*}
  \lim_{\varepsilon\rightarrow 0}\sup_{t\in[0,T]}\int_\Omega\left| |u_\varepsilon|^{(p-2)/2}u_\varepsilon -|u|^{(p-2)/2}u\right|^2dx=0.
\end{equation*}
Namely,
\begin{align*}
  0&=\lim_{\varepsilon\rightarrow 0}\int_\Omega\left| |u_\varepsilon(x,t)|^{(p-2)/2}u_\varepsilon (x,t) -|u(x,t)|^{(p-2)/2}u(x,t)\right|^2dx\\
  &=\lim_{\varepsilon\rightarrow 0}
  \int_\Omega |u_\varepsilon(x,t)|^pdx
  +\int_\Omega |u(x,t)|^pdx \\
  &\qquad\qquad\qquad-2\lim_{n\rightarrow\infty}\int_\Omega|u_\varepsilon(x,t)|^{(p-2)/2}u_\varepsilon(x,t)  |u(x,t)|^{(p-2)/2}u(x,t)dx\\
  &=\lim_{\varepsilon\rightarrow 0}
  \int_\Omega |u_\varepsilon(x,t)|^pdx
  -\int_\Omega |u(x,t)|^pdx,
\end{align*}
for any $t\in (0,T)$.
This together with $u_\varepsilon\stackrel{*}{\rightharpoonup} u$ in $L^\infty((0,T); L^p(\Omega))$ imply \eqref{eq:4.6}.

3. Now we show that $u\in L^\infty((0,T); W^{2,p}(\Omega))$ and $\theta=\Delta u$.
In fact, by the weak convergence of $\Delta_{NL}^{J_\varepsilon}\widetilde{u}_\varepsilon$,
\begin{equation}\label{eq:4.11}
  \int_{\Omega_E}\Delta_{NL}^{J_\varepsilon}\widetilde{u}_\varepsilon\cdot{\varphi} dx\rightarrow \int_{\Omega_E}\theta\cdot{\varphi} dx,
\end{equation}
as $\varepsilon\rightarrow 0$, for any $t\in (0,T)$ and any $\varphi\in C^\infty(\overline{\Omega_E})$.
By integration by parts and \eqref{eq:4.4},
\begin{equation*}
  \int_{\Omega_E}\Delta_{NL}^{J_\varepsilon}\widetilde{u}_\varepsilon\cdot{\varphi} dx=\int_{\Omega_E}\widetilde{u}_\varepsilon \cdot \Delta_{NL}^{J_\varepsilon}\hat{\varphi} dx
  =\int_{\Omega}u_\varepsilon \cdot \Delta_{NL}^{J_\varepsilon}\hat{\varphi} dx
  \rightarrow \int_{\Omega}u\cdot\Delta\varphi dx,
\end{equation*}
where $\hat{\varphi}$ is any smooth extension of $\varphi$.
Substituting it into \eqref{eq:4.11}, we have
\begin{equation}\label{eq:4.12}
  \int_{\Omega}u\cdot\Delta\varphi dx= \int_{\Omega}\theta\cdot\varphi dx,
\end{equation}
for any $t\in (0,T)$ and any $\varphi\in C^\infty(\overline{\Omega})$.
Since $u(\cdot, t)\in W_0^{1,\min\{2,p\}}(\Omega)$, identity \eqref{eq:4.12} says that $u(\cdot,t)$ is the weak solution of the Poisson equation $\Delta u=\theta$ with homogeneous Dirichlet boundary condition.
According to the $L^p$ theory of the Poisson equation, we conclude that $u\in L^\infty((0,T); W^{2,p}(\Omega))$ and $\theta=\Delta u$ for a.e. $x\in\Omega$  and a.e. $t\in (0,T)$.

We go back to \eqref{eq:4.12} to see that
\begin{equation*}
  \int_{\Omega}u\cdot\Delta\varphi dx=\int_{\Omega}\Delta u\cdot\varphi dx+ \int_{\partial\Omega}u\frac{\partial\varphi}{\partial\vec{n}}dS- \int_{\partial\Omega}\varphi\frac{\partial u}{\partial\vec{n}}dS
  =\int_{\Omega}\theta\cdot\varphi dx=\int_{\Omega}\Delta u\cdot\varphi dx.
\end{equation*}
Recall that $u=0$ on $\partial\Omega$. We have
\begin{equation*}
  \int_{\partial\Omega}\varphi\frac{\partial u}{\partial\vec{n}}dS=0,
\end{equation*}
for any $\varphi\in C^\infty(\overline{\Omega})$.
Namely,
\begin{equation*}
  \frac{\partial u}{\partial\vec{n}}=0,\quad \mbox{on } \partial\Omega.
\end{equation*}

4.
At last, we pass to the limit $\varepsilon\rightarrow 0$ in \eqref{eq:4.7}.
For the second term in \eqref{eq:4.7},
\begin{align*}
  \int_0^{t}\int_{\Omega_{E}}&|\Delta_{NL}^{J_\varepsilon}\widetilde{u}_\varepsilon|^{p-2}
  \Delta_{NL}^{J_\varepsilon}\widetilde{u}_\varepsilon\cdot
  \Delta_{NL}^{J_\varepsilon}\widetilde{\varphi} dxd\tau \\ &=\left(\int_0^{t}\int_{\Omega}+\int_0^{t}\int_{\Omega_{E}\backslash\Omega}\right)
  |\Delta_{NL}^{J_\varepsilon}\widetilde{u}_\varepsilon|^{p-2}
  \Delta_{NL}^{J_\varepsilon}\widetilde{u}_\varepsilon\cdot
  \Delta_{NL}^{J_\varepsilon}\widetilde{\varphi} dxd\tau,
\end{align*}
Notice that the second integral on the right hand side
\begin{equation*}
  \leq \int_0^{t}\int_{\Omega_\varepsilon}
  \left(\frac{p-1}{p}|\Delta_{NL}^{J_\varepsilon}\widetilde{u}_\varepsilon|^p
  +\frac{1}{p}|\Delta_{NL}^{J_\varepsilon}\widetilde{\varphi}|^{p}\right)dxd\tau\rightarrow 0,
\end{equation*}
as $\varepsilon\rightarrow 0$.
Here we use the fact that $\Omega_\varepsilon:=\{\Omega+\varepsilon\mbox{supp}(J)\}\backslash\Omega\rightarrow\emptyset$ as $\varepsilon\rightarrow 0$.

Passing to the limit $\varepsilon\rightarrow 0$ in \eqref{eq:4.7} yields
\begin{equation}\label{eq4.10}
  -\iint_{Q_T}u\frac{\partial \varphi}{\partial t}dxdt +\iint_{Q_T}\vartheta\cdot
  \Delta{\varphi} dxdt =0,
\end{equation}
for any $\varphi\in C_0^\infty(Q_T)$.
We are left to show that
 \begin{equation*}
   \iint_{Q_T}\vartheta
  \Delta{\varphi} dxdt= \iint_{Q_T}|\Delta u|^{p-2}\Delta u \Delta\varphi dxdt.
 \end{equation*}
Again, we utilize the monotonicity of $|s|^{p-2}s$.
This is similar to the proof of Theorem \ref{th:existence}. We omit it here.
\end{proof}

Having established the nonlocal-to-local convergence for the homogeneous equation \eqref{eq:1.1}, we now consider the nonhomogeneous case \eqref{eq:nonhomogeneous}.
Assume that the boundary value $g\in W^{2,p}(\Omega_E\backslash\overline{\Omega})$. There exists an extension of $g$ (still denoted by itself) such that $g\in W^{2,p}(\Omega_E)$.
Define $g_1=g$ and $g_2=\frac{\partial g}{\partial\vec{n}}$ on $\partial\Omega$.
The corresponding local equation with the non-homogeneous Dirichlet boundary condition reads
\begin{align}\label{eq:4.21}
\left\{
  \begin{array}{ll}
  \frac{\partial u}{\partial t}=-\Delta(|\Delta u|^{p-2}\Delta u), \quad (x,t)\in Q_T, \\
  u(x,t)=g_1(x),\quad \frac{\partial u}{\partial\vec{n}} (x,t)=g_2(x),  \quad (x,t)\in \partial\Omega\times (0,T), \\
  u(x,0)=u_0(x),\quad x\in\Omega.
  \end{array}
\right.
\end{align}

It admits a unique solution in the following sense.

\begin{lemma}
  Let $u_0\in L^2(\Omega)$ and $g\in W^{2,p}(\Omega)$ such that $u_0-g\in W_0^{2,p}(\Omega)$. Then equation \eqref{eq:4.21} admits a unique weak solution
  $u\in C([0,T]; L^2(\Omega))$, such that
  $u-g\in L^\infty((0,T); W_0^{2,p}(\Omega))$,
\begin{equation*}
    -\iint_{Q_T}u\frac{\partial \varphi}{\partial t}dxdt +\iint_{Q_T}|\Delta u|^{p-2}\Delta u \Delta\varphi dxdt =0,
\end{equation*}
for any $\varphi\in C_0^\infty(Q_T)$ and $u(x,0)=u_0(x)$ a.e. in $\Omega$.
\end{lemma}

The nonlocal-to-local convergence for equation \eqref{eq:nonhomogeneous} is as follows.
\begin{corollary}
    Let $1<p<\infty$, $u_0\in L^2(\Omega)$, and $g\in W^{2,p}(\Omega_E)$ such that $u_0-g\in W_0^{2,p}(\Omega)$.
    If $u_\varepsilon$ is the solution of the nonlocal equation \eqref{eq:nonhomogeneous} and $u$ is the weak solution of the local equation \eqref{eq:4.21}, then
    \begin{equation*}
    \lim_{\varepsilon\rightarrow 0}\sup_{t\in[0,T]}\|u_\varepsilon(x,t)-u(x,t)\|_{L^p(\Omega)}=0.
  \end{equation*}
\end{corollary}
\begin{proof}
Let $v_\varepsilon = u_\varepsilon -g$.
According to Corollary \ref{cor:1},
\begin{equation*}
  \int_0^{t}\int_{\Omega}\frac{\partial v_\varepsilon}{\partial t}\varphi dxd\tau +\int_0^{t}\int_{\Omega_{E}}|\Delta_{NL}^{J_\varepsilon}(\widetilde{v}_\varepsilon+g)|^{p-2}
  \Delta_{NL}^{J_\varepsilon}(\widetilde{v}_\varepsilon+g)\cdot
  \Delta_{NL}^{J_\varepsilon}\widetilde{\varphi} dxd\tau =0,
\end{equation*}
for any $t\in (0,T)$ and any $\varphi\in L^2(\Omega)\cap L^p(\Omega)$.
By repeating the proof of Theorem \eqref{th:4.1}, we obtain that 
    \begin{equation*}
    \lim_{\varepsilon\rightarrow 0}\sup_{t\in[0,T]}\|v_\varepsilon(x,t)-v(x,t)\|_{L^p(\Omega)}=0,
  \end{equation*}
  where $v\in L^\infty((0,T); W_0^{2,p}(\Omega))$ satisfies
\begin{equation*}
    -\iint_{Q_T}v\frac{\partial \varphi}{\partial t}dxdt +\iint_{Q_T}|\Delta (v+g)|^{p-2}\Delta (v+g) \Delta\varphi dxdt =0,
\end{equation*}
for any $\varphi\in C_0^\infty(Q_T)$. 
Since $\iint_{Q_T}g\frac{\partial \varphi}{\partial t}dxdt=0$,
we deduce that $u=v+g$ is the weak solution of equation \eqref{eq:4.21}.
This finishes the proof.
\end{proof}

\section{Numerical experiments}
In this section, we consider image inpainting as an example to illustrate the application of the nonlocal $p$-biharmonic equation \eqref{eq:nonhomogeneous} for image processing.
Let $D$ denote the image domain and $\Omega_d$ denote the inpainting region.
For any image $u$, let $u_i$ denote the value at pixel $i\in D$.
We are given partially observed data $g$ on $D\backslash \Omega_d$.
Equation \eqref{eq:nonhomogeneous} is discretized as
\begin{align}\label{eq:5.1}
\left\{
  \begin{array}{ll}
  \frac{u_i^{n+1}-u_i^n}{\tau}=-\Delta_{NLd}(|\Delta_{NLd}u_i^n|^{p-2}\Delta_{NLd}u_i^n), \quad i\in\Omega_d, n\geq 0, \\
  u_i^n=g_i,  \quad i\in D\backslash \Omega_d, n\geq 0, \\
  u_i^0=u_{0,i},\quad i\in\Omega_d.
  \end{array}
\right.
\end{align}
Here $u_0$ is an initial guess satisfying $u_0=g$ on $D\backslash\Omega_d$, and $u_i^n$ denotes the restored image at time level $n$ and pixel $i$.
The nonlocal Laplacian $\Delta_{NLd}$ of $u$ at pixel $i$ is defined by
\[
\Delta_{NLd}u_i=\sum_{j\in \mathcal{S}_i}w_{i,j}(u_j-u_i),
\]
where $w_{i,j}\geq 0$ represents the weight between pixels $i$ and $j$, and $\mathcal{S}_i\subset D$ denotes the search window centered at pixel $i$.
We adopt the classical nonlocal framework to compute the weights $w_{i,j}$.
Specifically,
\begin{equation}\label{eq:5.2}
  w_{i,j}=\frac{1}{Z_{i}}e^{-\frac{\|u_0(\mathcal{N}_{i})-u_0(\mathcal{N}_{j})\|^2_{2,a}}{2h^2}},
\end{equation}
where $Z_{i}$ is a normalization constant satisfying $\sum_{j\in \mathcal{S}_i}w_{i,j}=1$, and
$\|u_0(\mathcal{N}_i)-u_0(\mathcal{N}j)\|_{2,a}$ denotes the Gaussian-weighted Euclidean distance measuring the similarity between two patches $\mathcal{N}_i$ and $\mathcal{N}_j$ centered at pixels $i$ and $j$, respectively.
Here, $h>0$ is a scaling parameter.

For comparison, we also consider the nonlocal Laplacian (NLL) equation
\begin{align*}
\left\{
  \begin{array}{ll}
  \frac{u_i^{n+1}-u_i^n}{\tau}=\Delta_{NLd}u_i^n, \quad i\in\Omega_d, n\geq 0, \\
  u_i^n=g_i,  \quad i\in D\backslash \Omega_d, n\geq 0, \\
  u_i^0=u_{0,i},\quad i\in\Omega_d.
  \end{array}
\right.
\end{align*}
In the implementation, we use a $21\times 21$ search window and an $11\times 11$ patch size.
For each pixel, only the 10 most similar neighbors are retained for weight computation.
The parameter $h$ is set to $5$.
The initial guess $u_0$ is set to zero on $\Omega_d$.
Under this initialization, the weights $w_{i,j}$ defined in \eqref{eq:5.2} may fail to accurately reflect the nonlocal similarity.
To overcome this limitation, we first compute a solution from the zero initialization and subsequently update $u_0$ using the solution of the nonlocal equation.
This procedure is iteratively repeated 10 times.

Figure \ref{fig:1} presents the inpainting results of NLL and the nonlocal biharmonic equation (NLB) (i.e., equation \eqref{eq:5.1} with $p=2$) for the test image Barbara.
In the partially observed images (the second column of Figure \ref{fig:1}), the yellow regions indicate the inpainting domain (i.e., $\Omega_d$ in equation \eqref{eq:5.1}).
In the second row of Figure \ref{fig:1}, the mask is constructed by randomly selecting $80\%$ of the pixels from the image.
From the results we observe that the fourth-order nonlocal model NLB consistently outperforms the second-order nonlocal model NLL in terms of both visual quality and PSNR.
As discussed in \cite{wen2023nonlocal}, higher-order models exhibit improved contrast preservation, which enables them to better maintain structural information and produce finer reconstruction results.

In Figure \ref{fig:2}, we illustrate the performance of equation \eqref{eq:5.1} for different values of $p$.
It can be observed that as $p$ increases, the restored images become smoother, with fewer spike-like artifacts. This observation is consistent with the theoretical results.
Specifically, in the local counterpart \eqref{eq:4.21}, the solution space is $W^{2,p}(\Omega)$, and larger values of $p$ correspond to smoother solutions.

\begin{figure}[tbp]
  \centering
  \begin{tabular}{@{}c@{~}c@{~~~}c@{~}c@{}}
  \begin{overpic}[width=.24\textwidth]{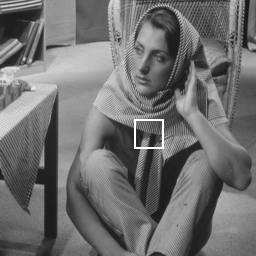}\put(61,0){\includegraphics[width=.095\textwidth]{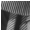}}\end{overpic}&
  \begin{overpic}[width=.24\textwidth]{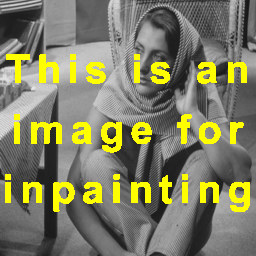}\end{overpic}&
  \begin{overpic}[width=.24\textwidth]{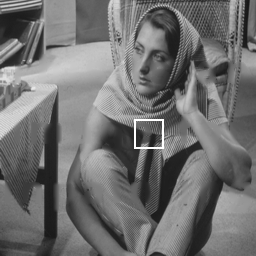}\put(61,0){\includegraphics[width=.095\textwidth]{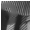}}\end{overpic}&
  \begin{overpic}[width=.24\textwidth]{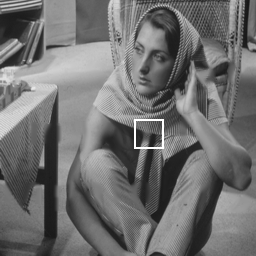}\put(61,0){\includegraphics[width=.095\textwidth]{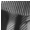}}\end{overpic}\\
  \parbox[c]{.24\textwidth}{\centering \footnotesize\emph{Barbara, $256\times256$}} &
  \parbox[c]{.24\textwidth}{\centering \footnotesize\emph{mask}} &
  \parbox[c]{.24\textwidth}{\centering \footnotesize\emph{NLL, 34.96dB}} &
  \parbox[c]{.24\textwidth}{\centering \footnotesize\emph{NLB, 37.01dB}} \\
  \begin{overpic}[width=.24\textwidth]{barbara_ori.png}\put(61,0){\includegraphics[width=.095\textwidth]{barbara_s.png}}\end{overpic}&
  \begin{overpic}[width=.24\textwidth]{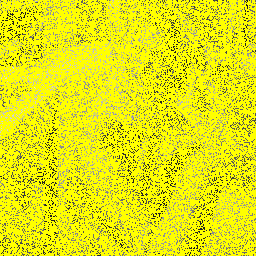}\end{overpic}&
  \begin{overpic}[width=.24\textwidth]{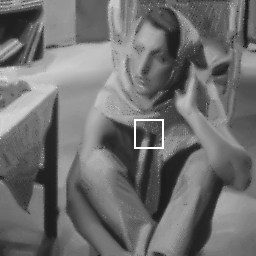}\put(61,0){\includegraphics[width=.095\textwidth]{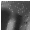}}\end{overpic}&
  \begin{overpic}[width=.24\textwidth]{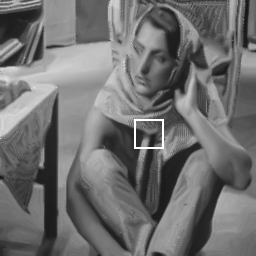}\put(61,0){\includegraphics[width=.095\textwidth]{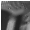}}\end{overpic}\\
  \parbox[c]{.24\textwidth}{\centering \footnotesize\emph{Barbara, $256\times256$}} &
  \parbox[c]{.24\textwidth}{\centering \footnotesize\emph{mask, $80\%$ pixels}} &
  \parbox[c]{.24\textwidth}{\centering \footnotesize\emph{NLL, 26.92dB}} &
  \parbox[c]{.24\textwidth}{\centering \footnotesize\emph{NLB, 28.39dB}} 
  \end{tabular}
\caption{Image inpainting results of the nonlocal Laplacian equation and the nonlocal biharmonic equation.}
\label{fig:1}
  \end{figure}

\begin{figure}[tbp]
  \centering
  \begin{tabular}{@{~}c@{~}c@{~}c@{~}c@{~}c@{}}
  \begin{overpic}[width=.19\textwidth]{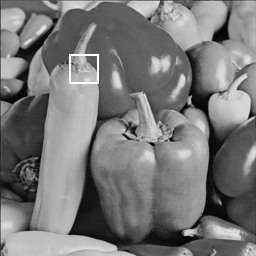}\put(48,0){\includegraphics[width=.1\textwidth]{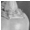}}\end{overpic}&
  \begin{overpic}[width=.19\textwidth]{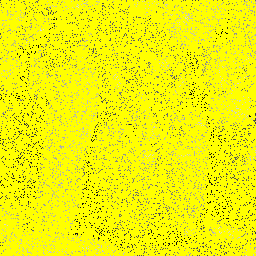}\end{overpic}&
  \begin{overpic}[width=.19\textwidth]{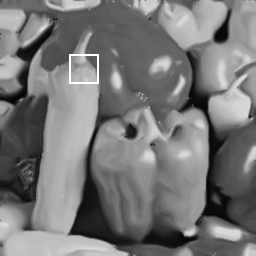}\put(48,0){\includegraphics[width=.1\textwidth]{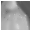}}\end{overpic}&
  \begin{overpic}[width=.19\textwidth]{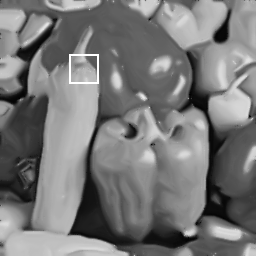}\put(48,0){\includegraphics[width=.1\textwidth]{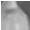}}\end{overpic}&
  \begin{overpic}[width=.19\textwidth]{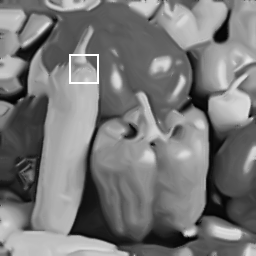}\put(48,0){\includegraphics[width=.1\textwidth]{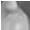}}\end{overpic}\\
  \parbox[c]{.19\textwidth}{\centering \footnotesize\emph{Peppers,$256\times256$}} &
  \parbox[c]{.19\textwidth}{\centering \footnotesize\emph{mask, $90\%$ pixels}} &
  \parbox[c]{.19\textwidth}{\centering \footnotesize\emph{$p=1.8$, 24.77dB}} &
  \parbox[c]{.19\textwidth}{\centering \footnotesize\emph{$p=2$, 25.27dB}}  &
  \parbox[c]{.19\textwidth}{\centering \footnotesize\emph{$p=2.5$, 25.63dB}}
  \end{tabular}
    \caption{Image inpainting results of the nonlocal $p$-biharmonic equation with different $p$.
  }
  \label{fig:2}
  \end{figure}

\section{Conclusion}
In this paper, we established the existence and uniqueness of solutions for a nonlocal $p$-biharmonic evolution equation with Dirichlet boundary conditions. We further showed that the solutions converged to those of the classical $p$-biharmonic equation, thereby providing a rigorous connection between the nonlocal and local models. The results were extended to non-homogeneous Dirichlet boundary conditions. Numerical experiments validated the effectiveness of the nonlocal model in image inpainting tasks.

\section*{Data available}
The data that support the findings of this study are available from the corresponding author upon reasonable request.

\section*{Conflict of interest}
The authors declare no potential conflict of interests.

\section*{Acknowledgment}
There are no funders to report for this submission.


\bibliographystyle{amsplain}

\bibliography{reference}

\end{document}